\documentclass[12pt]{article}
\usepackage{amsmath,amssymb,amsthm,fullpage}
\usepackage[numbers]{natbib}
\usepackage{bbm}
\usepackage{color}
\definecolor{ablue}{rgb}{0.09,0.32,0.44} 
\usepackage[pdfborder={0 0 0},pdfborderstyle={},colorlinks=true,linkcolor=ablue,citecolor=ablue,urlcolor=blue]{hyperref}

\allowdisplaybreaks[1]

\newcommand{\Z}{{\mathbb Z}}

\newcommand{\R}{{\mathbb R}}

\newcommand{\EE}{\,{\cal E}}
\newcommand{\RR}{{\mathbb R}}
\newcommand{\vphi}{{\varphi}}
\newcommand{\one}{\mathbbm{1}}
\newcommand{\nottriag}{\not\!\!\Delta}

\usepackage{braket}

\newcommand{\abs}[1]{\left\vert #1 \right\vert}
\newcommand{\norm}[1]{\left\Vert #1 \right\Vert}
\newcommand{\parens}[1]{\left( #1 \right)}

\theoremstyle{plain}
\newtheorem{theorem}{Theorem}[section]
\newtheorem{corollary}[theorem]{Corollary}

\newtheorem{lemma}[theorem]{Lemma}

\theoremstyle{remark}

\theoremstyle{definition}
\newtheorem{definition}{Definition}[section]

\usepackage{color}
\definecolor{red}{rgb}{.8,0,0}
\definecolor{green}{rgb}{0,.7,0}
\definecolor{blue}{rgb}{0,0,.8}

\DeclareMathOperator{\Ric}{Ric}
\DeclareMathOperator{\sgn}{sgn}
\DeclareMathOperator{\dist}{dist}

\def\lam{\lambda}

\title{Discrete curvature and abelian groups}
\author{Bo'az Klartag\thanks{Research supported in part by the
    European Research Council.}, Gady Kozma\thanks{Research supported
    in part by the Israel Science Foundation and the Jesselson Foundation.}, Peter Ralli\thanks{Research supported in part by the NSF grant DMS-1407657.} and Prasad Tetali\footnotemark[3]}
\date{August 27, 2015}
\begin{document}
\maketitle

\begin{abstract}
We study a natural discrete Bochner-type inequality on graphs, and explore its merit as a notion of ``curvature" in discrete spaces.
An appealing feature of this discrete version of the so-called
$\Gamma_2$-calculus (of Bakry-\'Emery) seems to be that it is fairly
straightforward to compute this notion of curvature parameter for
several specific graphs of interest -- particularly, abelian groups, slices of the hypercube, and the symmetric group under various sets of generators.
We further develop this notion by deriving  Buser-type inequalities (\`a la Ledoux), relating functional and isoperimetric constants associated with a graph.
Our derivations provide a tight  bound on the Cheeger constant (i.e., the {\em edge-isoperimetric constant})   in terms  of the spectral gap, for graphs with nonnegative curvature, particularly, the class of abelian Cayley graphs -- a result of independent interest.
\end{abstract}
\section{Introduction}
For several decades now it has been a fruitful endeavour to translate
notions from Riemannian geometry to graph theory. It is now clear what
are the graph analogs of the laplacian, Poincar\'e inequality, Harnack
inequality, and many related notions. The graph point of view led to
generalizations which would have been less natural in Riemannian
geometry, such as $\beta$-parabolic Harnack inequalities (see, e.g., \cite{BB04}), and to some counterexamples \cite{B04,D02,K14}.

Despite all this progress, the graph analog of the notion of curvature
remained elusive. In their 1985 paper, Bakry and \'Emery \cite{BE85}
suggested a notion analogous to curvature that would work in the very
general framework of a Markov semigroup (which, of course,
incorporates both continuous diffusions and random walks on
graphs). The condition was based on the Bochner formula
and was denoted by $CD(K,\infty)$ (for curvature-dimension) where $K$ is a curvature parameter. A semigroup satisfying $CD(K,\infty)$ is a generalization of Brownian
motion on a manifold with Ricci curvature $\ge K$ and hence the
condition $CD(K,\infty)$ is often called simply ``$\Ric\ge K$''
and we will stick to this convention in this paper. This notion as a possible definition of ``Ricci curvature" in Markov chains was in fact considered and discussed in \cite{Sch99} in 1999, but seems to have  largely been neglected ever since. For additional and more recent approaches to discrete Ricci curvature and related inequalities, see \cite{Bauer2013, EM12, GRST13, LV09, O09, Sam05, S06}. The fact that one can conclude from positive (or negative) curvature, a
local property, global facts about the manifold, has inspired similar
``local-to-global'' principles in group theory. See e.g.\ \cite{G73, P96}.

Beyond lower bounds on curvature, the proofs in \cite{BE85} (and in the recent book  \cite{BGL14}) rely on two additional assumptions on the semigroup. The
first was the existence of an appropriate algebra of smooth
functions. The second was a chain-rule formula for the
generator of the semigroup. A generator
satisfying the latter assumption is called a \emph{diffusion operator}, see
\cite[Definition 1.11.1, page 43]{BGL14}. In continuous setting it is
actually the existence of the required algebra of smooth functions
that is the most difficult condition to verify, but in graph settings this condition holds
immediately. Nevertheless, the diffusion condition can never hold in
the discrete setting.

However, the diffusion condition is not
always necessary. Denote the Cheeger constant (sometimes known as the
isoperimetric constant) by $h$, the spectral gap by $\lambda$ and
recall the inequality of Buser \cite{Buser82} that states that for a manifold with
non-negative Ricci curvature $\lambda \le 9 h^2$ (exact definitions will be given
in the next section). In 2006, the first two authors
noted that the arguments of Ledoux~\cite{Ledoux04spectralgap}, allow
to derive a discrete Buser-type inequality just assuming non-negative
Ricci curvature.

\begin{theorem}\label{thm:Buser}
A graph satisfying $\Ric\ge 0$ satisfies that
  $\lambda \le 16h^2$.
\end{theorem}

The graph version of Cheeger's inequality (e.g. \cite{ALS07, Ch10}), which does not require positive curvature, states that $\lambda \geq h^2 / (2d)$, where $d$ is the maximum degree of the graph. Thus for graphs with non-negative Ricci and bounded degree we get that $\lambda \approx h^2$.
As the
results from 2006 were never published, we include them in \S
\ref{sec:Buser}. A preprint of these results did circulate and a
number of papers built on it \cite{Bauer2013, LP14}. Particularly
relevant for us is the paper \cite{LP14} which shows that
the \emph{eigenvalues} of the laplacian on a graph with positive
curvature satisfy $\lambda_k\le Ck^2\lambda_1$.
In a similar spirit, we use the techniques of
\cite{Ledoux04spectralgap} to show a Gaussian type isoperimetric inequality
for graphs satisfying $\Ric\ge 0$ (see Section~\ref{iso-LS} below).

In light of Theorem~\ref{thm:Buser}, an intriguing and challenging open problem is to characterize the class of graphs with non-negative Ricci curvature.
The main new results of this paper are examples of such graphs which
satisfy $\Ric\ge 0$. These include Cayley graphs of abelian groups, the complete graph, the group $S_n$ with all transpositions, and slices of
the hypercube.


In particular, we get Buser's inequality for any Cayley graph of a
finite abelian group.
We remark that this is not true for a general
group. For example, the Cayley graph of the group $S_n$ with the
generators being $\{(12),(12\dotsc n)^{\pm 1}\}$ has $h$ of order $1/n^2$ and $\lambda \ge 1/n^3$, up to an absolute constant (we fill some details about these well-known facts in \S\ref{subsec:abelian}). This should be compared against the fact that
any compact Lie group has positive Ricci curvature, see
\cite[Corollary 3.19, page 65]{CE75}.

Note that our results above translate to $\lam(M) \le 16 d \ h^2(M)$
for a simple random walk $M$ on an abelian Cayley graph, regular of
degree $d$, with $h(M)$ and $\lam(M)$ being defined for the Markov
chain version. 
A result of the above type is also recently derived
independently by Erbar and by Oveis-Gharan and Trevisan (private communications). An earlier, weaker result, $\lam(M) = O(d^2 \ h^2(M))$  follows from the work in \cite{Bauer2013}, which uses a different notion of curvature (and a different argument of Ledoux), starting from a finite-dimensional curvature-dimension $CD(K, n)$ inequality for graphs.

Recently there have been several attempts to modify the $CD(K,n)$ criterion in order to allow certain results involving the heat equation \cite{Bauer2013, HLLY15, Munch15}.  A recent result of M\"unch \cite{Munch15} is that the $CDE'(K,n)$ criterion of \cite{HLLY15} implies the $CD(K,n)$ criterion of Bakry-\'{E}mery.  These criteria are often useful; for example, it is known that Ricci-flat graphs satisfy both the $CDE(0,\infty)$ criterion of \cite{Bauer2013} and the $CDE'(0,\infty)$ criterion.

In the remainder of this section, we introduce Bochner's $\Gamma_2$-type curvature for graphs along with various notations and definitions. In Section~\ref{sec:examples}, we bound the curvature for several examples, including slices of the discrete cube, symmetric group with adjacent as well as all transpositions as the generating sets;  
and nonnegativity of curvature for Cayley graphs of abelian groups. In Section~\ref{sec:gap}, we show that the spectral gap can be bounded from below by curvature.  In Section~\ref{sec:Buser},
we derive the above-mentioned Buser-type inequalities.

\subsection{Preliminaries}

We first recall some basic definitions and fairly standard notions.
Let $G=(V,E)$ be an undirected and locally finite graph.  Throughout, we will assume that $G$ has no isolated vertices.  The graph Laplacian $\Delta = \Delta(G)  = -(D(G) - A(G))$, where $D(G)$ is the diagonal matrix of the degrees of the vertices, and $A(G)$ is the adjacency matrix of $G$.  As an operator, its action on an $f:V\to \R$ can be described as:
\[ \Delta f(x) = \sum_{y\sim x} (f(y)-f(x))\,.\]
where here and below the notation $y\sim x$ means that $y$ is a
neighbour of $x$ in the graph. The sum is of course only over the
$y$. Note that $\Delta$ is a negative semi-definite matrix.

The spectral gap $\lambda(G)$ is the least non-zero eigenvalue of $-\Delta$.  We define the Cheeger constant \[h(G) = \min_{0 < |A| \leq |V|/2} \frac{|\partial A|}{|A|},\] where $|\partial A|$ denotes the number of edges from $A$ to $V-A$.

 Given functions $f, g: V\to \R$, we also define:
\[\Gamma(f,g)(x) = \frac{1}{2}\sum_{y\sim x} (f(x)-f(y))(g(x)-g(y))\,. \]
When $f=g$, the above becomes the more commonly denoted (square of the $l_2$-type) discrete gradient: for each $x\in V$,
\[\Gamma(f)(x):=\Gamma(f,f)(x) = \frac{1}{2}\sum_{y\sim x} (f(x)-f(y))^2 =:|\nabla f(x)|^2\,. \]

It becomes useful to define the iterated gradient
\[2 \Gamma_2(f,g) = \Delta\Gamma(f,g) - \Gamma(f,\Delta g) - \Gamma(\Delta f, g)\,.\]
By convention, \[\Gamma_2(f):=\Gamma_2(f,f) = \frac{1}{2}\Delta\Gamma(f) - \Gamma(f,\Delta f).\]
Note that, given a measure $\pi:V\to [0,\infty)$, one can consider the expectation (with respect to $\pi$) of the above quantity, which gives us the more familiar Dirichlet form associated with a graph:
\[\EE(f,g):= \frac{1}{2}\sum_x \sum_{y\sim x} (f(x)-f(y))(g(x)-g(y)) \pi(x)\,.\]
It is useful to note an identity: \begin{equation}\label{eq:global}\sum_{x\in V}\Gamma(f,g)(x) = -\sum_{x\in V}f(x)\Delta g(x) = -\sum_{x\in V}g(x)\Delta f(x).\end{equation}
An  additional useful local identity is:
\begin{equation}\label{eq:local1}
\triangle (f g) = f \triangle g + 2 \Gamma( f, g ) + g \triangle f,
\end{equation}

\noindent

\begin{definition}
The (Bochner) curvature $\text{Ric}(G)$ of a graph $G$ is defined as the maximum value $K$ so that, for any function $f$ and vertex $x$, we have
\begin{align}\label{bochsmall} \Gamma_2(f)(x) \geq K\Gamma(f)(x)\,.
\end{align}
\end{definition}

Let $x\in V$, and let $f:V\to \R$ be a function.  Observe that (\ref{bochsmall}) is unchanged on adding a constant to $f$, so we may assume that $f(x) = 0$.  We expand $\Gamma_2(f)(x)$:
\begin{align}
\lefteqn{2\Gamma_2(f)(x) = \Delta\Gamma(f)(x) - 2\Gamma(f,\Delta f)(x)}
\quad & \nonumber\\
& = \sum_{v\sim x} \Gamma(f)(v) - d(x)\Gamma(f)(x) - \sum_{v\sim x} f(v)\parens{\Delta f(v)-\Delta f(x)}
\nonumber \\
& = \frac{1}{2}\sum_{u\sim v\sim x} \parens{f(u)-f(v)}^2 - \frac{d(x)}{2}\sum_{v\sim x}f^2(v) + \sum_{v\sim x}f(v)\sum_{u\sim x} f(u) -\sum_{u\sim v\sim x} f(v) \parens{f(u)-f(v)}
\nonumber \\
& = \parens{\sum_{v\sim x}f(v)}^2 - \frac{d(x)}{2}\sum_{v\sim x}f^2(v) + \sum_{u\sim v\sim x}\frac{f^2(u) - 4f(u)f(v) + 3f^2(v)}{2}
\nonumber \\
& = \parens{\sum_{v\sim x}f(v)}^2 - \sum_{v\sim x}\frac{d(x) + d(v)}{2} f^2(v) + \frac{1}{2}\sum_{u\sim v\sim x} \parens{f(u)-2f(v)}^2\,.\label{eq:uvx}
\end{align}
Now, we break the latter term into the cases that $u = x$, $u\sim x$ and $d(x,u) = 2$.  In the second case, we denote by $\Delta(x,v,u)$ the set of all unordered pairs $(u,v)$ satisfying $x\sim u\sim v\sim x$.  The above is equal to
\begin{align}
2\Gamma_2(f)&=\frac{1}{2}\sum_{\substack{u\sim v\sim x\\ d(x,u) = 2}} \parens{f(u)-2f(v)}^2 + \parens{\sum_{v\sim x}f(v)}^2 + \sum_{v\sim x}\parens{2-\frac{d(x) + d(v)}{2}}f^2(v)
\nonumber\\
&\qquad +\;\sum_{\Delta(x,v,u)}\frac{\parens{f(v)-2f(u)}^2 + \parens{f(u)-2f(v)}^2}{2}
\nonumber \\
& = \frac{1}{2}\sum_{\substack{u\sim v\sim x\\ d(x,u) = 2}}\parens{f(u)-2f(v)}^2 + \parens{\sum_{v\sim x}f(v)}^2 + \sum_{v\sim x} \frac{4-d(x)-d(v)}{2}f^2(v)
\nonumber\\
&\qquad +\;\sum_{\Delta(x,v,u)}\biggl[2\parens{f(v)-f(u)}^2 + \frac{1}{2}\parens{f^2(v) + f^2 (u)}\biggr]\,. \label{boch}
\end{align}
Fixing $f(v)$ for all vertices $v\sim x$, we may ask what choice of $f(u)$ (for $d(x, u) = 2$)
minimizes the above expression?  We wish to minimize \begin{align*}\frac{1}{2}\sum_{\substack{v:\\ x\sim v\sim u}} \parens{f(u)-2f(v)}^2,\end{align*} it is simple to see that the minimizer is
\begin{align}\label{umin} f(u) = 2\cdot \frac{1}{r(u)} \sum_{x\sim v \sim u} f(v)\,,\end{align}
where $r(u)$ is the number of common neighbors of $u$ and $x$.

We first prove a general upper bound on the above notion of curvature, which will be used in the next section, to show tightness of our bounds on curvature for several example graphs.
\begin{theorem}\label{ricupper} Let $G = (V,E)$ be a graph.  If $e\in E$, let $t(e)$ denote the number of triangles containing $e$.  Define $T:=\max_{e} t(e)$.  Then $\Ric(G) \leq 2 + \frac{T}{2}$.\end{theorem}

\begin{proof}   Let $x\in V$ be any vertex with the minimum degree
  $d$, and consider the distance (to $x$) function $f(v) = \text{dist}(v,x)$.
  It is simple to calculate that
\[2\Gamma_2(f)(x)
\stackrel{\textrm{(\ref{boch})}}{=}
d^2 + \sum_{v\sim x} \parens{2 - \frac{d + \deg(v)}{2}} + \sum_{\Delta(x,v,u)} 1 \leq 2d + \frac{dT}{2},\] observing that \[|\Delta(x,v,u)| = \frac{1}{2}\sum_{v\sim x}t(x,v)\leq \frac{dT}{2}\] and that $\Gamma(f)(x) = \tfrac{1}{2}d$.  Any value of $K > 2 + \frac{T}{2}$ will not satisfy (\ref{bochsmall}) for the function $f$ at vertex $x$, thus $\Ric(G) \leq 2 + \frac{T}{2}$.\end{proof}

\section{Examples}
\label{sec:examples}

In this section we provide bounds on the curvature for several graphs of general interest.

\subsection{The hypercube \texorpdfstring{$H_n$}{}}

Let $H_n$ represent the $n$-dimensional hypercube, where vertices are
adjacent if their Hamming distance is one. While the following result
also follows from the tensorization result of \cite{Sch99}, we provide here a direct proof.

\begin{theorem} $\Ric(H_n) = 2$ if $n\geq 1$. \end{theorem}

\begin{proof}
For any vertex $x \in H_n$, and for any $f$ with $f(x)=0$, we get from (\ref{boch})
\begin{align*}2\Gamma_2(f)(x) = \frac{1}{2}\sum_{\substack{u:\\ d(x,u) = 2}}\sum_{\substack{v:\\ x\sim v\sim u}}\parens{f(u)-2f(v)}^2 + \parens{\sum_{v\sim x}f(v)}^2 + (2-n) \sum_{v\sim x} f^2(v).\end{align*}

Let $u$ be a vertex of distance $2$ from $x$, and let $v$ and $w$ be the two distinct vertices so that $u\sim v\sim x\sim w\sim u$.  Then for fixed values of $f(v)$ where $v\sim x$, according to (\ref{umin}) $\Gamma_2(f)(x)$ is minimized by $f(u) = f(v) + f(w)$.  With this value, \[\sum_{v:u\sim v\sim x} \parens{f(u)-2f(v)}^2 = 2\parens{f(v)-f(w)}^2.\]
As for every pair $v,w\sim x$ there is a unique vertex $u$ with $u\sim
v,w$ and $d(x,u) = 2$,
\begin{align*} 2\Gamma_2(f)(x)
 \geq & \sum_{\substack{v\neq w\\ v,w\sim x}} \parens{f(v)-f(w)}^2 +
 \parens{\sum_{v\sim x}f(v)}^2 + (2-n) \sum_{v\sim x} f^2(v)\,,
\end{align*} where the first sum is over all unordered pairs $(v,w)$ of distinct neighbors of $x$.  We use this convention throughout the paper.
Expanding the above gives
\begin{align*}
&\sum_{\substack{v\neq w\\ v,w\sim x}} (f^2(v) + f^2(w)) - \sum_{\substack{v\neq w\\ v,w\sim x}} 2f(v)f(w) +  \sum_{v\sim x} f^2(v) + \sum_{\substack{v\neq w\\ v,w\sim x}} 2f(v)f(w) + (2-n)\sum_{v\sim x} f^2(v) & \\
  & =  2\sum_{v\sim x} f^2(v) = 4\Gamma(f)(x)\,.
\end{align*}
So $\text{Ric}\geq 2$, and by Theorem~\ref{ricupper} we may conclude that $\Ric = 2$.
\end{proof}

 \

In the following, we compute the curvature of the complete graph. With
the tensorization result of \cite{LP14}, this provide another proof of the fact that the hypercube has curvature 2.

 \subsection{The complete graph \texorpdfstring{$K_n$}{}}

 \begin{theorem}
 $\Ric(K_n) = 1 + \frac{n}{2}$ if $n\geq 2$.
 \end{theorem}

\begin{proof}
For the complete graph on $n$ vertices, we have, for every $x\in V$
and every $f:V\to\RR$ such that $f(x)=0$, from (\ref{boch}),
\begin{multline*}
2\Gamma_2(f)(x) = \\
\Big(\sum_{v\sim x}f(v)\Big)^2
  + (3-n)\sum_{v\sim x} f^2(v)
  + \sum_{\substack{u,v\sim x\\ u\neq v}}\Big(2\parens{f(v)-f(u)}^2
      + \frac{1}{2}\parens{f(u)^2 + f(v)^2}\Big).
\end{multline*}
Expanding the above gives
\begin{align*}
\sum_{v\sim x} f^2(v)
  &+ \sum_{\substack{u,v\sim x\\ u\neq v}} 2f(u)f(v) + (3-n)\sum_{v\sim x} f^2(v)
  + \frac{5}{2}\sum_{\substack{u,v\sim x\\ u\neq v}}(f^2(v) + f^2(u))
  - \sum_{\substack{u,v\sim x\\ u\neq v}}4f(u)f(v)\\
&= (4-n)\sum_{v\sim x} f^2(v) + \frac{5}{2}(n-2)\sum_{v\sim x} f^2(v)
  - 2\sum_{\substack{u,v\sim x\\ u\neq v}}f(u)f(v)\\
&= \parens{\frac{3n}{2}-1}\sum_{v\sim x}f^2(v)
  - 2\sum_{\substack{u,v\sim x\\ u\neq v}}f(u)f(v) = \frac{3n}{2}\sum_{v\sim x}f^2(v)
  - \Big(\sum_{v\sim x}f(v)\Big)^2\,.
 \end{align*}

By the Cauchy-Schwarz inequality, $\parens{\sum_{v\sim x}f(v)}^2 \leq \abs{\set{v:v\sim x}}\sum_{v\sim x}f^2(v) = (n-1)\sum_{v\sim x}f^2(v)$, so
\begin{align*}\frac{3n}{2}\sum_{v\sim x}f^2(v) - \parens{\sum_{v\sim x}f(v)}^2 \geq \parens{1+\frac{n}{2}}\sum_{v\sim x}f^2(v)\,.\end{align*}
Thus $\text{Ric}\geq 1 + \frac{n}{2}$, once again by Theorem~\ref{ricupper}, we conclude that $\Ric = 1 + \frac{n}{2}$.
\end{proof}

\subsection{Finite abelian Cayley graphs}\label{subsec:abelian}

A finite abelian group is of course a product of cyclic groups and hence one
might think that the curvature of the graph can be deduced from the
tensorization result of \cite{LP14}. However, a Cayley graph is determined by an underlying group and a generating set for that group. Here we show that a finitely generated abelian
group with \emph{any} set of generators has positive Ricci
curvature - not only with the generating set inherited from a
decomposition into cyclic groups. This result was implicit in the
literature, since abelian Cayley graphs are ``Ricci flat'' \cite{CY86},
and this property, in turn, gives $\Ric\ge 0$ \cite{LY}. We give here a
direct proof.

Let us remark that the problem of graphs locally identical
to an abelian group has also been attacked successfully using
combinatorial tools. See \cite{BE} and references within.

\begin{theorem}\label{cayley} Let $X$ be a finitely generated abelian group, and $S$ a finite set of generators for $X$.  Let $G$ be the Cayley graph corresponding to $X$ and $S$.  Then $\Ric(G)\geq 0$.\end{theorem}

Recall that the Cayley graph of a group $G$ with respect to a given
set $S$ which generates $G$ is the graph whose vertices are the
elements of $G$ and whose edges are $\{(g,gs)\}_{g\in G, s\in
  S}$. Since we are interested in undirected graphs, $S$ should be
symmetric i.e.\ $s\in S\Rightarrow s^{-1}\in S$.
\begin{proof}Without loss of generality, we may set $x$ to be the
  identity element of $X$.  Denote the degree of every vertex by
  $d$. As usual, let $f:G\to\RR$ with $f(x)=0$.

For this calculation, we prefer not to distinguish between $u$
according to their distance from $x$ so we start the calculation from
(\ref{eq:uvx}) and using the constant degree get
\begin{align}\label{bochother}
2\Gamma_2(f)(x) = d\sum_{v\sim x} f^2(v)
  + \Big(\sum_{v\sim x} f(v)\Big)^2
  + \sum_{v\sim x} \sum_{u\sim v} \parens{\frac{f^2(u)}{2} - 2f(u)f(v)}.
\end{align}
Because $x$ is the identity, we observe that if $u\sim v\sim x$, there is a unique $w\sim x$ so that $u = vw$.  We can express the last term of (\ref{bochother}) as
\begin{align*}&\sum_{v\sim x} \sum_{u\sim v} \parens{\frac{f^2(u)}{2} - 2f(u)f(v)} = \sum_{v\sim x}\sum_{w\sim x}\parens{ \frac{f^2(vw)}{2} - 2f(vw)f(v)}\\
&= \sum_{v\sim x} \parens{\frac{f^2(v^2)}{2}-2f(v^2)f(v)} + \sum_{\substack{v,w\sim x\\ v\neq w}}\parens{ f^2(vw) - 2f(vw)\parens{f(v) + f(w)}} \\
&\geq -2\sum_{v\sim x}f^2(v) - \sum_{\substack{v,w\sim x\\ v\neq w}} \parens{f(v) + f(w)}^2 = (-d-1)\sum_{v\sim x}f^2(v) - 2\sum_{\substack{v,w\sim x\\ v\neq w}} f(v)f(w).\end{align*}
In the last passage we used the elementary inequalities $a^2/2 - 2ab \geq -2b^2$ and $a^2 - 2ab \geq -b^2$

Plugging this bound into (\ref{bochother}), we find that
\begin{align*} 2\Gamma_2(f)(x)\geq  \parens{\sum_{v\sim x} f(v)}^2 -\sum_{v\sim x} f^2(v) - 2\sum_{\substack{v,w\sim x\\ v\neq w}} f(v)f(w) = 0.
\end{align*}
This completes the proof.
\end{proof}

Now, the assumption that the group is abelian is necessary.
An infinite example demonstrating this is the $d$-ary tree, which is the Cayley graph of the group $\langle s_1,...,s_d:s_i^2 = id\text{ for }i = 1,...,d\rangle$ with the generating set $s_1,...,s_d$. This graph has $\Ric = 2-d$, which is achieved whenever $\sum_{y\sim x} f(y) = 0$ and $f(z) = 2f(y)$ whenever $z\sim y\sim x$.  This is optimal; it is not difficult to see that no $d$-regular graph has $\Ric(G) < 2-d$.

 A little more surprising, perhaps, is that the Heisenberg group also has negative curvature. We mean here the group of upper triangular matrices with 1 on the diagonal and \emph{integer entries}, equipped with the set of generators $\left\{\left(\begin{smallmatrix} 1 & \pm 1 & 0 \\
& \hphantom{\pm}1 & 0\\
&&1\end{smallmatrix}\right),
\left(\begin{smallmatrix} 1 & 0 & \hphantom{\pm}0 \\
& 1 & \pm 1\\
&&\hphantom{\pm}1\end{smallmatrix}\right)\right\}$. It is straightforward to check that these generators do not satisfy any relation of length 4, so the environment within distance 2 (which is the only relevant distance for calculation of the curvature) is tree-like, and the curvature would be $-2$.

Switching to finite Cayley graphs, it is well-known that there exist finite Cayley graphs which are locally tree-like, and hence would have negative curvature. What is perhaps more interesting is that even
Buser's inequality (the conclusion of
Theorem \ref{thm:Buser1}) may fail.

\begin{theorem}For the group $S_n$ and the (left) Cayley graph generated by
  $\{(12),(12\dotsc n)^{\pm 1}\}$, the Cheeger constant is $\le c_1
  n^{-2}$, while the spectral gap is $\ge c_2n^{-3}$, with $c_1, c_2 > 0$, independent of $n$.
\end{theorem}
\begin{proof}[Proof sketch] To show an upper bound on  the Cheeger constant, we consider
  the following set:
\[
A = \{\phi\in S_n:\dist(\phi(1),\phi(2))\le \tfrac14 n\}
\]
(there is no connection between the 1 and 2 in the definition of $A$
and the fact that we took $(12)$ as a generator). Here dist is the
cyclic distance between two numbers in $\{1,\dotsc,n\}$
i.e.\ $\min(|x-y|,n-|x-y|)$. Clearly $|A|=(\frac12 +o(1))n!$. To
calculate the size of the boundary we first note that the generators
$(12\dotsc n)^{\pm 1}$ keep $A$ invariant, so the boundary of $A$ is
composed of edges between $\phi\in A$ and $(12)\phi\not\in A$. This
makes two requirements on $\phi$: first it must satisfy that
$\dist(\phi(1),\phi(2))=\lfloor \frac14 n\rfloor$, and second it must
satisfy that one of $\phi(1),\phi(2)$ is in the set $\{1,2\}$
otherwise the application of $(12)$ does nothing to $\phi(1)$ and
$\phi(2)$ and $(12)\phi$ would still be in $A$. Thus $\partial A
\approx n!/n^2$ and $h\ge c/n^2$ (this argument gives $c=2+o(1)$).

The estimate of the spectral gap (from below) for the random walk on this Cayley graph was done by Diaconis and Saloff-Coste (see Section 5.3 in \cite{DSC95}), as an example of the comparison argument -- comparing with the random transposition chain, which has a spectral gap of order $1/n$, gives a lower bound of $(1/10)n^{-3}$ for this chain; since the graph has a bounded degree, the spectral gap of the graph laplacian is only a constant factor off that of the random walk on the graph.

For the convenience of the reader, and for completeness, we now sketch a proof of a lower bound of $1/(n^3 \log n)$, which serves to justify the point of the theorem. We construct a
\emph{coupling} between two lazy random
walkers on our group $S_n$ that succeeds by time $n^3\log n$. It is
well-known (see e.g.\ \cite{LPW09}) that this bounds the mixing time,
and hence the relaxation time, which is the inverse of the spectral
gap. The coupling is as follows: assume $\phi_n$ and $\psi_n$ are our
two walkers. We apply exactly the same random walks steps to $\phi_n$
except in one case: when for some $i$ $\phi_n(i)=1$ and $\psi_n(i)=2$.
In this case when we apply a $(12)$ step for $\phi_n$ we apply a lazy
step to to $\psi_n$, and vice versa (the $(12\dotsc n)^{\pm 1}$ are
still applied together). It
is easy to check that for each $i$, $\phi_n(i)-\psi_n(i)$ is
doing a random walk on $\{1,\dotsc,n\}$, slowed down by a factor of
$n$, with gluing at $0$. Therefore it glues with positive probability
by time $n^3$ and with probability $>1-1/2n$ by time $Cn^3\log
n$. Thus by this time, with probability $>\frac 12$ we have
$\phi(i)=\psi(i)$ for all $i$, or in other words, the coupling
succeeded. This shows that the mixing time is $\le Cn^3\log n$ and in turn
gives a lower bound on the spectral gap.
\end{proof}


\subsection{Cycles and infinite path}

We consider the cycle $C_n$ for $n\geq 3$.  We extend the notation by letting $C_{\infty}$ denote the infinite path.

From previous results it is simple to observe that $\Ric(C_3) = \frac{5}{2}$, as $C_3 = K_3$, and that $\Ric(C_4) = 2$, because $C_4 = H_2$.
\begin{theorem}\label{cycles} If $n\geq 5$, $\Ric(C_n) = 0$.\end{theorem}

\begin{proof}

We note that the calculation of $\Ric(G)$ at $x$ requires us to consider only the subgraph consisting of those vertices $v$ with $d(x,v)\leq 2$, and those edges incident to at least one neighbor of $x$.

If $n\geq 5$, this subgraph will always be a path of length $4$ centered at $x$, so we only need calculate the curvature for this graph.  $C_n$ is an abelian Cayley graph, thus $\Ric\geq 0$.

$\Ric = 0$ is achieved by the function $f$ that takes values $-2,-1,0,1,2$ in order along the path.
\end{proof}

\begin{corollary} Let $\Z^d$ represent the infinite $d$-dimensional lattice.  $\Ric(\Z^d) = 0$.\end{corollary}

We simply note that $\Z^d$ is the product of $d$ copies of $C_\infty$.

\subsection{Slices of the hypercube}
\subsubsection{\texorpdfstring{$k$}{k}-slice with transpositions}

For some fixed value $k$ with $1\leq k < n$, let $G=(V,E)$ be the graph with $V = \{x\in \{0,1\}^n:\sum_i x_i = k\}$, and $x\sim y$ whenever $\abs{ \text{supp}(x-y)} = 2$.

\begin{theorem}This graph has curvature $\Ric = 1 + \frac{n}{2}$.
\end{theorem}
\begin{proof}
Let $x\in V$.  Define $s_{ij}x$ to be the vertex obtained by exchanging coordinates $i$ and $j$ in $x$.
A vertex $u$ with $d(x,u) = 2$ will be $u = s_{ij}s_{lm}x$ for some distinct coordinates $i,j,l,m$ with $x_i = x_l = 1$, $x_j = x_m = 0$.  Vertices $v$ with $x\sim v\sim u$ are $s_{ij}x, s_{im}x, s_{lj}x, s_{lm}x$.  Observe that \[\sum_{v:x\sim v\sim u} \parens{f(u)-2f(v)}^2\geq 2\parens{f(s_{ij}x)-f(s_{lm}x)}^2 + 2\parens{f(s_{im}x)-f(s_{lj}x)}^2.\]
Summing over all vertices $u$ with $d(x,u) = 2$ gives \[\frac{1}{2}\sum_{\substack{x\sim v\sim u\\ d(x,u) = 2}} \parens{f(u)-2f(v)}^2\geq \sum_{\substack{v,w\sim x\\ \not\Delta(x,v,w)}} \parens{f(v)-f(w)}^2,\] as for each pair $v,w\sim x$ with $v\not\sim w$, there is exactly one $u$ with $v,w\sim u$ and $d(x,u) = 2$.  (Here we use the notation $\nottriag(x,v,w)$ to denote the set of unordered pairs $(v,w)$ of distinct neighbors of $x$ for which $v\not\sim w$.)

Also notice that any $v\sim x$ has $t(\{x,v\}) = n-2$: if $v = s_{ij}x$, the vertices that make a triangle with $x$ and $v$ are $s_{lj}x$ when $l\neq i$ and $x_l = x_i$, and $s_{im}x$ when $m\neq j$ and $x_m = x_j$.

Now we may compute
\begin{align*}
\lefteqn{2\Gamma_2(f)(x)}\qquad&\\
& \geq \sum_{\substack{v,w\sim
      x\\ \not\Delta(xvw)}} \parens{f(v)-f(w)}^2 + \parens{\sum_{v\sim
      x}f(v)}^2 + \parens{2-d + \frac{n-2}{2}}\sum_{v\sim x}f(v)^2 \\
&\qquad + 2\sum_{\Delta(vwx)}\parens{f(v)-f(w)}^2\\
&\geq \sum_{v,w\sim x} \parens{f(v)-f(w)}^2 + \parens{\sum_{v\sim x}f(v)}^2 + \parens{1-d+ \frac{n}{2}}\sum_{v\sim x}f(v)^2\\
& = (d-1)\sum_{v\sim x}f(v)^2 -2\sum_{v,w\sim x} f(v)f(w) +
  \sum_{v\sim x}f(v)^2 + 2\sum_{v,w\sim x}f(v)f(w) \\
&\qquad + (1-d + \frac{n}{2})\sum_{v\sim x}f(v)^2\\
&=  \parens{1 + \frac{n}{2}}\sum_{v\sim x}f(v)^2\,.
\end{align*}
So $\text{Ric}(G)\geq 1 + \frac{n}{2}$. Together with Theorem~\ref{ricupper} we get that $\Ric = 1 + \frac{n}{2}$.
\end{proof}
\subsubsection{Middle slice with adjacent transpositions}
We now consider $G$ with $V = \{x\in \{-1,1\}^{2n}:\sum_i x_i = 0\}$, where $x\sim y \iff \text{supp}(x-y)$ consists of $2$ consecutive elements.  Alternately, $V$ is the set of paths in $\Z^2$ that move from $(0,0)$ to $(2n,0)$ with steps of $(+1,+1)$ and $(+1,-1)$, and paths $x$ and $y$ are neighbors if y can be achieved by transposing an adjacent $(+1,+1)$ and $(+1,-1)$ in $x$.

\begin{theorem}$\Ric(G) \geq -1$.  Further, $\displaystyle \lim_{n\to \infty} \Ric(G) = -1$. \end{theorem}

\begin{proof}
Let $x\in V$.  Let $I(x) = \{i\in \{1,\dots,2n-1\}: x_i\neq x_{i+1}\}$, so $i\in I$ if and only if we are allowed to switch segments $i$ and $i+1$.  If $i\in I(x)$, denote by $a_ix$ the vertex obtained by making this switch.  Observe $\abs{I(x)} = \text{deg}(x)$.

The neighbors of $a_i x$ are: $a_i(a_i x) = x$, $a_j(a_i x)$ for any $j\in I(x)$ with $|i-j|>1$, and $a_j(a_i x)$ for any $j\notin I(x)$ with $|i-j| = 1$ and $j\neq 0,2n$.  We calculate that $\text{deg}(a_i x) = \text{deg}(x) + 2 - 2\#\{j\in I(x): |i-j| = 1\} - \mathbbm{1}_{i = 1}-\mathbbm{1}_{i=2n-1}$.

We observe that a neighbor of the form $a_j(a_i x)$ if $j\in I(x)$ and $|i-j|>1$ will be identical to $a_i(a_j x)$, and have $d(x,a_ja_ix) = 2$.

Now, for any function $f$,
\begin{align*}& \frac{1}{2}\sum_{\substack{u\sim v\sim x\\ d(x,u) = 2}} \parens{f(u)-2f(v)}^2 \\ \geq & \frac{1}{2}\sum_{\substack{i,j\in I\\ |i-j|>1}} \parens{f(a_ia_j x)-2f(a_i x)}^2 +\parens{f(a_ia_j x)-2f(a_j x)}^2\\ \geq &\sum_{\substack{i,j\in I\\ |i-j|>1}} \parens{f(a_i x)-f(a_j x)}^2 \\
= & \sum_{i\in I(x)}\#\{j\in I(x):|j-i| > 1\}f^2(a_i x)-2\sum_{\substack{i,j\in I\\ |i-j|>1}} f(a_i x)f (a_j x)\,.
\end{align*}
Observe that $G$ is triangle-free.  We have that
\begin{align*}&2\Gamma_2(f)(x)\\ \geq & \sum_{i\in I(x)}\#\{j\in I(x):|j-i| > 1\}f^2(a_i x)-2\sum_{\substack{i,j\in I\\ |i-j|>1}} f(a_i x)f (a_j x)\\
+ &\sum_{i\in I(x)} f^2(a_i x) + 2\sum_{i,j\in I} f(a_i x)f(a_j x)\\
+ & \sum_{i\in I(x)}\parens{2-\frac{2\cdot \text{deg}(x) + 2 - 2\#\{j\in I(x): |i-j| = 1\} - \mathbbm{1}_{i = 1}-\mathbbm{1}_{i=2n-1}}{2}}f^2(a_i x)\\
\geq & \sum_{i\in I(x)}\parens{\#\{j\in I(x):i\neq j\} + 2 - \text{deg}(x)}f^2(a_i x) + 2\sum_{\substack{i,j\in I\\ |i-j|=1}} f(a_i x) f(a_j x)\\
 = &\sum_{i\in I(x)} f^2(a_i x) + 2\sum_{\substack{i,j\in I\\ |i-j|=1}} f(a_i x) f(a_j x) \\
 > & -\sum_{i\in I(x)}f^2(a_i x) + \sum_{\substack{i,j\in I\\ |i-j|=1}} \parens{f(a_i x)+f(a_j x)}^2 \geq -2\Gamma(f)(x)\,.
\end{align*}

So $\text{Ric}(G) > -1$, where we ignore a slight  dependence on $n$ in the lower order term.

Define a function with $f(+1,-1,+1,-1,...) = 0$ and $f(a_i x) = f(x) - x_i$, that is, if the switch lowers the path, $f$ decreases by $1$; a switch that raises the path will increase $f$ by $1$.

Using this $f$ and $x = (+1,-1,+1,-1,...)$, we find that $\text{Ric} \to -1$ as $n\to \infty$.
\end{proof}

We now calculate the curvature for the subgraph $G_+$ that is induced on the Dyck paths, i.e., those paths that are always on or above the $x$-axis. Alternately, sequences in $\{\pm 1\}^{2n}$ with $\sum_{i = 1}^{2n} x_i = 0$ and $\sum_{i = 1}^j x_i \geq 0$ for all $j = 0,...,2n$.  It is well-known that the number of Dyck paths is the Catalan number $C_n$.

\begin{corollary}For this subgraph $G_+$, $\Ric(G_+) \geq -1$.  Further, $\displaystyle \lim_{n\to \infty} \Ric(G_+) = -1$. \end{corollary}

\begin{proof}[Proof sketch]
Let $x\in V$, and let
\[
I(x) = \{i\in [2n-1]:\text{ a possible move is to transpose } x_i,
x_{i+1}\}.
\]
If $i\in I$, let $a_ix$ be the sequence obtained by transposing $x_i,x_{i+1}$.

Observe that $\text{deg}(a_i x)\leq \text{deg}(x) + 2 - 2\#\{j\in I(x): |i-j| = 1\} - \mathbbm{1}_{i = 1}-\mathbbm{1}_{i=2n-1}$.  Using the same analysis as in the unrestricted problem, we may conclude that
\begin{align*}&2\Gamma_2(f)(x) \geq -2\Gamma(f)(x).
\end{align*}

A similar test-function as above will prove that $\text{Ric} \leq  -1
+ o(1)$.  We may use the same function $f$, and take $x$ identical to
the above example but with the first $-1$ and last $+1$ transposed.
This will give a similar upper bound on Ric. (Observe that the
neighbors and second-neighbors of $x$ in the unrestricted graph are
all Dyck paths, so the curvature at $x$ will be unchanged from the
original.)
\end{proof}

\subsection{The symmetric group \texorpdfstring{$S_n$}{} with all transpositions}
\begin{theorem}\label{symmetric} Let $G$ be the Cayley graph on the
  symmetric group $S_n$ with all transpositions as generators.  Then $\Ric(G) = 2$.\end{theorem}

Let us remark that in recent work \cite{EMT14}
the authors also provided a lower bound for the Ricci curvature of the
(Cayley) graph on the symmetric group with the edge set given by
transpositions, but with a different notion of
Ricci curvature, one developed by Erbar and Maas \cite{EM12}.  It is easy to see that the Ricci curvature developed by Ollivier \cite{O09} gives a value of $\kappa = 2/\binom{n}{2}$ for this problem in the setting of a Markov chain. A simple coupling argument shows that this agrees with our result, modulo the normalizing factor between the graph setting and the Markov chain setting.

\begin{proof}
Let $x\in S_n$.  A vertex $u$ with $d(u,x) = 2$ will either be $(ijk)x$ for some distinct $i,j,k\in [n]$ or $(ij)(kl)x$ for distinct $i,j,k,l\in [n]$.

In the first case, the vertices $v$ s.t. $(ijk)x \sim v\sim x$ are $v = (ij)x, (ik)x, (jk)x$.
For $u = (ijk)(x)$,
\begin{align*}
\lefteqn{\sum_{v:u\sim v\sim x} \parens{f(u)-2f(v)}^2}\qquad &\\
&= \big(f(u)-2f((ij)x)\big)^2 + \big(f(u)-2f((ik)x)\big)^2 + \big(f(u)-2f((jk)x)\big)^2 \\
& \geq \frac{4}{3}\left[\big(f((ij)x)-f((ik)x)\big)^2 +
  \big(f((ij)x)-f((jk)x)\big)^2 + \big(f((ik)x)-f((jk)x)\big)^2\right].
\end{align*}
In the second case, a $v$ such that $(ij)(kl)x \sim v\sim x$ is either
$v = (ij)x$ or $v = (kl)x$.
If $u = (ij)(kl)(x)$,
\[\sum_{v : u\sim v\sim x} \big(f(u)-2f(v)\big)^2 \geq 2\big(f((ij)x) - f((kl)x)\big)^2.\]
Taking a sum over all values of $u$ gives \begin{align*}\frac{1}{2}\sum_{\substack{u\sim v\sim x\\ d(u,x) =2}} \parens{f(u)-2f(v)}^2 \geq \sum_{v,w\sim x} \parens{f(v)-f(w)}^2.\end{align*}
Indeed, if $v,w$ are $v = (ij)x$ and $w = (ik)x$ for some $i,j,k$, the term $\parens{f(v)-f(w)}$ is counted twice in the sum: for $u = (ijk)x $ and $u = (ikj) x$.  If $v,w$ are $v = (ij)x$ and $w = (kl)x$ for some $i,j,k,l$, the term $2\parens{f(v)-f(w)}$ is counted once: for $u = (ij)(kl)x$.

Observe that $G$ is triangle-free and regular with degree $d = \binom{n}{2}$.
Using this bound, we see that
\begin{align*}
2\Gamma_2(f)(x)
  & \geq \sum_{v,w\sim x} \parens{f(v)-f(w)}^2 + \Big(\sum_{v\sim x}f(v)\Big)^2
     + (2-d)\sum_{v\sim x}f^2(v)\\
  & = 2\sum_{v\sim x}f^2(v) = 4\Gamma(f)(x).
\end{align*}
Therefore $\Ric \geq 2$, as $G$ is triangle-free, $\Ric = 2$ by Theorem \ref{ricupper}.
\end{proof}

\section{Spectral gap and curvature}\label{sec:gap}

Let $\lambda(G)$ denote the spectral gap of $G$; i.e., the least nonzero eigenvalue of $-\Delta$.

\begin{theorem}\label{thm:gap_curvature}
Let $G$ be a graph with curvature $\text{Ric} \geq K\geq 0$.  Then $\lambda \ge K$.
\end{theorem}

A different proof of this result was given in \cite{CLY14}.

\begin{proof}
We may use the 2nd derivative versus the first derivative (of variance of the heat kernel) characterization of the spectral gap (see e.g. \cite{MT06}).
\begin{align*} \lambda = \min_f\frac{\mathcal{E}(-\Delta f, f)}{\mathcal{E}(f,f)},\end{align*} so that $\alpha \leq \lambda$ if and only if, for any function $f$, we have $\alpha\cdot\mathcal{E}(f,f) \leq \mathcal{E}(-\Delta f,f)$.

By assumption, $G$ satisfies (\ref{bochsmall}) with parameter $K$, i.e., that $$\Delta\Gamma(f)(x) - 2\Gamma(f,\Delta f)(x) - 2K \Gamma(f)(x) \geq 0\,,$$
for all functions $f:V\to \R$ and all $x\in V$.
Summing the above inequality over all vertices gives
\begin{align*}
\lefteqn{\sum_x \Delta\Gamma(f)(x) - 2\sum_x\Gamma(\Delta f,f)(x)
    - 2K\sum_x \Gamma(f)(x)}\qquad&\\
 & = 2\sum_x(\Delta f(x))^2 - K\sum_x \sum_{y\sim x} \parens{f(y)-f(x)}^2\\
 & = 2\sum_x (\Delta f(x))^2 - 2K\sum_{x\sim y} \parens{f(y)-f(x)}^2 \geq 0
\end{align*}
where in the first equality, we used the identity (\ref{eq:global})
and the fact that for any $g$, $\sum\Delta g=0$.

Now let $|V|=n$, and recall the Dirichlet form (with respect to the
measure $\pi\equiv 1$),
\begin{align*}
  \mathcal{E}(f,f) = \sum_{x\sim y} (f(y)-f(x))^2
\end{align*}
and that
\begin{align*}
\mathcal{E}(-\Delta f, f) = \sum_x -\Delta f(x) \Big(\sum_{y\sim x} (f(x)-f(y))\Big)
  = \sum_x (\Delta f(x))^2\,.
\end{align*}
Plugging into the above inequality gives
\begin{align*}
  2\mathcal{E}(-\Delta f, f) - 2K\mathcal{E}(f,f)\geq 0,
\end{align*}
and so
\begin{align*}
  K\mathcal{E}(f,f)\leq \mathcal{E}(-\Delta f,f)\,,
\end{align*}
resulting in $\lambda \geq K$\,.
\end{proof}


\section{Buser-type Inequalities}\label{sec:Buser}

The proofs in this section are a straightforward discrete version of \S~5 of Ledoux's paper \cite{Ledoux04spectralgap}.
First we derive a key gradient estimate on the heat kernel associated with a graph, which will then be used in deriving a Buser inequality for graphs, as mentioned in the introduction.

\subsection{Gradient estimates}

For $t \geq 0$, we write $P_t = \exp(t \triangle)$ for the heat kernel associated with the graph $G$.
Then $P_t$ is a positive definite matrix on $\RR^V$,
with $P_0$ being the identity matrix. Note that $P_t$
commutes with $\triangle$ and with $P_s$,
and that $\partial P_t /
\partial t = P_t \triangle = \triangle P_t$.
Finally, the matrix $P_t$ has non-negative entries.
So if $f$ has non-negative entries, then
also $P_t(f)$ has non-negative entries.
For a vector $f: V \rightarrow \RR$
we write $\| f \|_p = (\sum_v |f(v)|^p)^{1/p}$.

\begin{lemma} Suppose $G$ has $\Ric(G) \geq K$ for some $K\in \R$. Then, for any $f: V \rightarrow
\RR$ and any $0 \leq t \leq 1/|2K|$,
$$ \| f - P_t f \|_1 \leq 2 \sqrt{t} \| \sqrt{\Gamma( f)}  \|_1 \,.$$
 \label{lem1}
\end{lemma}
Note that the restriction on $t$ applies only when $K$ is negative: if
$K>0$ then $\Ric\ge K$ implies $\Ric\ge 0$ and the lemma holds with no
restriction on $t$.

\begin{proof}The proof is in three steps.
\medskip

\noindent
{\bf Step 1.} We first prove that
$$\Gamma(P_tf) \leq e^{-2K t} P_t( \Gamma(f))\,, $$
where the inequality holds pointwise on $V$ (recalling that these are real-valued functions on $V$). To that end, define the auxiliary function $g_s =
e^{-2K s} P_s ( \Gamma( P_{t-s} f))$, a function on $V$. It is enough to show that $\partial g_s / \partial s$ is pointwise
non-negative on $(0, t)$. We compute
$$ \frac{\partial g_s}{\partial s} =  e^{-2K s} P_s \big[   2\Gamma_2(P_{t-sf})  - 2K\Gamma(P_{t-s} f)  \big]\,.
 $$
Since $P_s$ preserves non-negativity, it is enough to prove that
$$\Gamma_2(P_{t-sf})  - K\Gamma(P_{t-s} f) \geq 0,  $$ which is true by our assumption, that  $\Ric(G) \geq K$.

\medskip
\noindent
{\bf Step 2.} Next we prove that
\begin{equation}
 P_t(f^2) - (P_t f)^2 \geq \left( \int_0^t 2 e^{2K s} ds \right)\Gamma(P_t f).
 \label{eq_331}
 \end{equation}
To that end, define the auxiliary function
$g_s = P_s [ (P_{t-s} f)^2 ]$. It is enough to show that $\partial g_s /
\partial s \geq 2 e^{2K s} \Gamma(P_t f)$, for any $0 \leq s \leq t$.  We compute, using
the local identity (\ref{eq:local1}) mentioned earlier,
$$ \frac{\partial g_s}{\partial s} = P_s \left[
2 P_{t-s} f \cdot \triangle P_{t-s} f + 2 \Gamma(P_{t-s} f) \right]
+ P_s \left[   2 P_{t-s} f \cdot (-\triangle P_{t-s} f) \right]. $$
Hence, by Step 1, for any $0 \leq s \leq t$,
$$ \frac{\partial g_s}{\partial s} = 2 P_s \left( \Gamma( P_{t-s} f) \right) \geq 2 e^{2K s} \Gamma( P_t f), $$
which gives (\ref{eq_331}).

Denote $c_K(t) = \int_0^t 2 e^{2K s} ds$. Then
$c_K(t) =  (e^{2K t}-1) /K$, for non-zero $K$, and $c_K(t) = 2t$
for $K = 0$. In both cases, $c_K(t) \approx 2 t$ for small $t > 0$.
For instance, $c_K(t) \geq t$ for $0 \leq t \leq 1/(2|K|)$. Hence
(\ref{eq_331}) gives, for $0 \leq t \leq 1/(2|K|)$,
\begin{equation}
 \max \sqrt {\Gamma( P_t f) } \leq \frac{1}{\sqrt{ t}} \max \sqrt{P_t(f^2)} \leq
\frac{1}{\sqrt{ t}} \max |f|. \label{eq_337}
\end{equation}

\medskip
\noindent
{\bf Step 3.}  As can be guessed by now,
we begin by writing
$$ P_t f - f = \int_0^t \frac{\partial P_s f}{\partial s} ds = \int_0^t P_s \triangle  f ds. $$
To prove the lemma, it suffices to show that $\|  P_s (\triangle f)
\|_1 \leq s^{-1/2} \| \sqrt{ \Gamma( f)} \|_1$ (since we have
$\int_0^t s^{-1/2} ds = 2 \sqrt{t}$). Let $\psi = \sgn(P_s(\Delta f))$.
Then,
\begin{eqnarray*}
\lefteqn{  \| P_s (\triangle f) \|_1 = \sum_{x \in V} P_s (\triangle f)(x) \cdot \psi =
\sum_{x\in V} \triangle f(x) \cdot P_s(\psi)(x)  = \sum_{x\in V} -\Gamma( f, P_s(\psi) )(x) } \\ & \leq & \sum_{x\in V} \sqrt{ \Gamma (f)(x) \cdot \Gamma( P_s(\psi))(x)}
\leq \|  \sqrt{ \Gamma( f) } \|_1 \cdot \max_{x\in V} \sqrt{\Gamma( P_s(\psi))(x)},
\end{eqnarray*}
and the desired inequality follows from (\ref{eq_337}), as $\max |\psi| = 1$.
\end{proof}

\subsection{Spectral gap and isoperimetry}
\label{iso-gap}

\begin{theorem}\label{thm:Buser1} Suppose $G$ has $\Ric(G) \ge K$, for some $K \in \R$.
Denote by $\lambda > 0$, the minimal
non-zero eigenvalue of $-\triangle$.
Then, for any subset $A \subset V$,
$$ |\partial A| \geq \frac{1}{2} \min \left \{ \sqrt{\lambda}, \frac{\lambda}{\sqrt{2|K|}} \right \}   |A| \left(1 - \frac{|A|}{|V|} \right) . $$
Here, by $\partial A$, we mean the collection of all edges
connecting $A$ to its complement.
\end{theorem}

As noted in the previous lemma, the term $\lambda /\sqrt{2|K|}$ is relevant only in the
case $K<0$.

\begin{proof}Apply the previous lemma to $f = \one_A$. Then $\Gamma(\one_A)$ is the function which associates with each $v \in V$, the
number of edges in $\partial A$ that are incident with $v$.
Consequently, for any $0 < t < 1/(2|K|)$,
$$ \| \one_A - P_t (\one_A) \|_1 \leq {2}\sqrt{ t} \cdot |\partial A|. $$
Note that $0 \leq P_t(\one_A) \leq 1$, hence the left-hand side may be written as follows:
$$ \| \one_A - P_t(\one_A) \|_1
  = \abs{A} - \sum_A P_t(\one_A) + \sum_{A^c} P_t(\one_A)
  = 2 \left[ \abs{A} - \sum_V \one_A \cdot P_t(\one_A) \right]
$$
Since $P_t$ is self-adjoint and $P_{t/2} P_{t/2} = P_t$, then,
$$  (1/2)\| \one_A - P_t(\one_A) \|_1  = \abs{A} - \| P_{t/2} (\one_A) \|_2^2
= \| \one_A \|_2^2 - \| P_{t/2} (\one_A) \|_2^2. $$

Let $\phi_i:1\leq i\leq n$ be the orthonormal eigenvectors of $\Delta$, and let $\lambda_i$ be the corresponding eigenvalues.  Let $\one_A = \sum a_i \vphi_i$ be the spectral decomposition
of $A$, with $\vphi_0 \equiv 1 / \sqrt{\abs{V}}$ and $a_0 = \abs{A} / \sqrt{\abs{V}}$.
Then
$P_{t/2} (\one_A) = \sum_i a_i e^{-\lambda_i t/2} \vphi_i$, and
hence
$$  (1/2) \| \one_A - P_t(\one_A) \|_1 = \sum_i (1 - e^{-\lambda_i t}) a_i^2
\geq (1 - e^{-\lambda t} ) \sum_{i \geq 1} a_i^2 = (1 - e^{-\lambda
t}) \left( \abs{A} - \frac{\abs{A}^2}{\abs{V}} \right). $$ To summarize,
for any $0 < t \leq 1/(2|K|)$,
$$ \abs{\partial A} \geq \frac{1 - e^{-\lambda t}}{ \sqrt{t}} \abs{A} \left(1 - \frac{\abs{A}}{\abs{V}} \right) . $$
If $\lambda \geq 2|K|$, we select $t = 1/\lambda \le 1/2|K|$, and deduce the theorem (use $(1 - 1/e)
> 1/2$). If $\lambda \leq 2|K|$, we take the maximal possible value,
$t = 1/(2|K|)$. Then $1 - e^{-\lambda /2|K|} \geq \lambda/(4|K|)$, and the
theorem follows.
\end{proof}


\begin{corollary}
Suppose  a graph $G$ has $\text{Ric}(G)\geq K$, for some  $K \ge 0$. Then
\[ h \ge \frac{1}{4}  \sqrt{\lambda}\,.
\]
\end{corollary}
\begin{proof}
As already explained, when $K\ge 0$ we may ignore the term
$\lambda/\sqrt{2|K|}$ in the minimum in
Theorem~\ref{thm:gap_curvature} and then the theorem gives 
\begin{align*}
\frac{|\partial A|\cdot |V|}{|A|\cdot |\overline{A}|}\geq \frac12\sqrt{\lambda},\end{align*}
and so we have
\[h\geq \frac{1}{4}\sqrt{\lambda}\,.\qedhere\]
\end{proof}



\subsection{Logarithmic Sobolev constant and isoperimetry}
\label{iso-LS}
We now prove an analogue of Theorem~5.3 from \cite{Ledoux04spectralgap}, relating the log-Sobolev constant $\rho$ to an isoperimetric quantity.
Consider the hypercontractive formulation of the log-Sobolev constant (see e.g., \cite{Gross75},\cite{DSC96}):
namely, define $\rho$ to be the greatest value so that whenever $1<r<q<\infty$ and $\displaystyle\sqrt{\frac{q-1}{r-1}}\leq e^{\rho t}$, then $$n^{-1/q} \norm{P_tf}_q\leq n^{-1/r}\norm{f}_r\,.$$


\begin{theorem}
\label{thm:iso_LSI}
Suppose $G$ has $\Ric(G) \geq K$ for some value $K\in \R$. Then for any subset $A\subset V$ with $|A|\leq |V|/2 = n/2$,
$$|\partial A|\geq \frac{1}{16}\min\parens{\sqrt{\rho},\frac{\rho}{\sqrt{2\abs{K}}}} |A| \log\frac{n}{|A|}.$$
\end{theorem}

\begin{proof}

As in the proof of the above Theorem~\ref{thm:Buser1}, we can observe
that
$$\sqrt{t}\frac{|\partial A|}{n}\geq \frac{|A|}{n}-\frac{\norm{P_{t/2}(\one_A)}_2^2}{n}\,,$$
if $0 < t < 1/(2\abs{K})$\,.
\noindent
Using the hypercontractivity property with $q =2$ and $r = 1 + e^{-2\rho t}$ gives that $$\frac{\norm{P_{t/2}(\one_A)}_2^2}{n} \leq \frac{\norm{\one_A}_r^2}{n^{2/r}} = \parens{\frac{|A|}{n}}^{2/r}\,.$$
Hence,
\begin{align*}
 \sqrt{t}\frac{|\partial A|}{n}
 \geq \frac{|A|}{n} - \frac{\norm{P_{t/2}(\one_A)}_2^2}{n}
 \geq \frac{|A|}{n}-\parens{\frac{|A|}{n}}^{2/r}.
\end{align*}
As $2/r \geq 1 + \rho t/4$, whenever $0\leq \rho t \leq 1$, and $|A|/n \leq 1$,
\begin{align}\label{tbound}
 \sqrt{t}\frac{|\partial A|}{n}
 \geq  \frac{|A|}{n}-\parens{\frac{|A|}{n}}^{1+\rho t/4}
 = \frac{|A|}{n}\parens{1-\parens{\frac{|A|}{n}}^{\rho t/4}}.
\end{align}
Let $t_0 = \min\parens{1/2\abs{K},1/\rho}$.  If $|A|/n < e^{-4}$, set $\displaystyle t = \frac{4t_0}{\log (n/|A|)}$.

Using this value of $t$ in (\ref{tbound}), we find
\begin{align*}
  \frac{\abs{\partial A}}{n}
    & \geq \frac{1}{\sqrt{t}}\frac{|A|}{n}(1-e^{-\rho t_0})\\
  & \geq \frac{1}{2\sqrt{t_0}}\frac{|A|}{n}(1-e^{-\rho t_0})
      \log\Big(\frac{n}{|A|}\Big)^{1/2}
    \geq \frac{1}{4} \rho \sqrt{t_0}\frac{|A|}{n}
      \parens{\log \frac{n}{|A|}}^{1/2}\,.
\end{align*}
On the other hand, if $e^{-4}\leq |A|/n\leq \frac{1}{2}$, use $t = t_0$ in (\ref{tbound}) to find:
\begin{align*}\frac{\abs{\partial A}}{n} \geq \frac{1}{\sqrt{t_0}}\frac{|A|}{n}\parens{1 - 2^{-\rho t_0/4}} \geq \frac{1}{8} \rho \sqrt{t_0}\,. \frac{|A|}{n}\geq \frac{1}{16} \rho \sqrt{t_0}\frac{|A|}{n}\parens{\log \frac{n}{|A|}}^{1/2}\,,\end{align*}
where, for the second inequality, we use $1-2^{-x} \ge x/2$, if $0\le x \le 1$. Hence,
\begin{align*}
\frac{\abs{\partial A}}{n}\geq \frac{1}{16} \rho\sqrt{\min\parens{\frac{1}{2|K|},\frac{1}{\rho}}}\frac{|A|}{n}\parens{\log\frac{n}{|A|}}^{1/2} \geq \frac{1}{16}\min{\parens{\sqrt{\rho},\frac{\rho}{\sqrt{2|K|}}}}\frac{|A|}{n}\parens{\log\frac{n}{|A|}}^{1/2}\,,
\end{align*}
proving the theorem.
\end{proof}

The optimality of the above theorem (in terms of the dependence on the parameters involved) remains open at this time; in particular, we do not have tight examples.
It is also natural to ask if the bound $\rho \geq K$ holds when $\Ric \geq K \geq 0$, similar to the bound on $\lambda$ in Theorem~\ref{thm:gap_curvature}.  In general this is not true, consider the complete graph on $n$ vertices.  We have seen that $\Ric = 1 + \tfrac{n}{2}$, and it is easy to see (by considering the characteristic function of a set as a test function) and is also well-known that $\rho = O(\tfrac{n}{\log n})$ (see e.g., \cite{MT06}).

It is however true that under a different notion of discrete curvature
for reversible Markov chains, one developed by Erbar and Maas, the
so-called {\em modified} logarithmic Sobolev constant, $\rho_0$, can
be lower bounded by the curvature, see  \cite{EM12}. Thus it is certainly interesting to explore whether an analog of Theorem~\ref{thm:gap_curvature} is true with $\rho_0$ in place of $\lambda$; recall here that $\rho_0$ captures the rate of decay of relative entropy of the Markov chain, relative to the equilibrium distribution, while $\rho$ captures the hypercontractivity property of the Markov kernel (see \cite{MT06} for additional information).

\subsection*{Acknowledgment} The last author thanks Matthias Erbar, Shayan
Oveis-Gharan, and Luca Trevisan for discussions on Buser inequality
for graphs. We thank the referee for bringing to our attention a lot of relevant literature we were not aware of, especially \cite{CY86,LY,CLY14,Munch15}.

\bibliographystyle{plain}
\bibliography{ricci}

\end{document}